\documentclass[]{article}

\usepackage [T2A] {fontenc}
\usepackage [utf8] {inputenc}
\usepackage[russian]{babel}

\usepackage{mathtools}
\mathtoolsset{showonlyrefs}

\usepackage{indentfirst} 
\usepackage[14pt]{extsizes}
\usepackage{amssymb}
\usepackage{amsmath}
\usepackage{amsthm}
\usepackage{pifont}
\usepackage{subfig}
\usepackage{graphicx}
\usepackage[colorinlistoftodos]{todonotes}
\usepackage{makecell}
\usepackage{multirow}
\usepackage{booktabs}
\usepackage[colorlinks=true, allcolors=blue]{hyperref}
\usepackage{algorithm}
\usepackage{algorithmic}
\usepackage{natbib}

\newtheorem{theorem}{Теорема}
\newtheorem{corollary}{Следствие}[theorem]
\newtheorem{lemma}{Лемма}

\newtheorem{assumption}{Предположение}
\newtheorem{example}{Пример}
\newtheorem{remark}{Замечание}

\usepackage{geometry}
\allowdisplaybreaks

\begin{document}
\renewcommand{\abstractname}{\vspace{-\baselineskip}}

$$
\\\\
$$


\begin{center}
\Large{\textbf{Эффективный метод с компрессией для распределенных и федеративных кокоэрсивных вариационных неравенств}}
\end{center}

\begin{center}
\textbf{Д.\,О.~Медяков$^{1,2}$, Г.\,Л.~Молодцов$^{1,2}$, А.\,Н.~Безносиков$^{1,2,3}$}

\textit{
\textsuperscript{1} Институт системного программирования Российской академии наук, Москва, Россия \\
\textsuperscript{2} Московский физико-технический институт (национальный исследовательский университет), Москва, Россия \\
\textsuperscript{3} Университет Иннополис, Иннополис, Россия
}


\end{center}

\begin{abstract}
\begin{center}
    \textbf{Нотация}
\end{center}
\noindent Вариационные неравенства как эффективный инструмент для решения прикладных задач, в том числе задач машинного обучения, в последние годы привлекают всё больше внимания исследователей. Области применения вариационных неравенств охватывают широкий спектр направлений — от обучения с подкреплением и генеративных моделей до традиционных приложений в экономике и теории игр. В то же время, невозможно представить современный мир машинного обучения без подходов распределенной оптимизации, которые позволяют значительно ускорить процесс обучения на огромных объемах данных. Однако, сталкиваясь с большими затратами на коммуникации между устройствами в вычислительной сети, научное сообщество стремится к разработке подходов, делающих вычисления дешевыми и стабильными. В этой работе исследуется техника сжатия передаваемой информации применимо к задаче распределенных вариационных неравенств. В частности, предлагается метод на основе продвинутых техник, исконно разработанных для задач минимизации. Для нового метода приводится исчерпывающий теоретический анализ сходимости для кокоэрсивных сильно монотонных вариационных неравенств. Проведенные эксперименты подчеркивают высокую производительность представленной техники и подтверждают практическую применимость.
\end{abstract}

\section{Введение}
Вариационные неравенства (ВН) привлекают внимание исследователей в различных областях уже более полувека \citep{browder1965nonexpansive}. В данной работе рассматривается общая постановка задачи вариационных неравенств на неограниченном множестве:
\begin{align}\label{eq:vi_setup}
    \text{Найти~} z^*\in\mathbb R^d \text{~такое, что~} F(z^*) = 0,
\end{align}
где $F:\mathbb R^d \rightarrow \mathbb R^d$ -- заданный оператор. Такая общая постановка покрывает множество известных задач. Приведем несколько примеров.
\begin{example}\label{ex:1}
    Классическая задача минимизации:
    \begin{align*}
        \underset{z\in\mathbb R^d}{\min}~ f(z),
    \end{align*}
    где $f(z)$ -- некоторая целевая функция. Для того, чтобы свести задачу \eqref{eq:vi_setup} к задаче минимизации, достаточно рассмотреть оператор вида $F(z) = \nabla f(z)$. Более того, для выпуклых функций $f(z)$ точка $z^*$ будет решением задачи \eqref{eq:vi_setup} тогда и только тогда, когда она будет решением этой задачи.
\end{example}
\begin{example}\label{ex:2}
    Задача поиска седловой точки (минимаксная задача):
    \begin{align*}
        \underset{x\in\mathbb R^d}{\min}\underset{y\in\mathbb R^d}{\max} ~g(x, y),
    \end{align*}
    где $g(z) = g(x, y)$ -- некоторая целевая функция. Для того, чтобы свести задачу \eqref{eq:vi_setup} к минимаксной задаче, достаточно рассмотреть оператор вида $F(z) = F(x, y) = \left[\nabla_x g(x, y), -\nabla_y g(x, y)\right]$. Более того, для выпуклых-вогнутых функций $g(x, y)$ точка $z^* = (x^*, y^*)$ будет решением задачи \eqref{eq:vi_setup} тогда и только тогда, когда она будет решением минимаксной задачи.
\end{example}

Тем не менее, несмотря на общность постановки, она практически неприменима в современных прикладных задачах, в первую очередь в задачах обучения. Дело в том, что считать полное значение оператора на одном вычислительном устройстве слишком затратно по времени из-за огромных объемов тренировочных данных. Поэтому на практике используют распределенные подходы \citep{kairouz2021advances,li2020federated,verbraeken2020survey}. В такой парадигме используется сеть из нескольких устройств, каждое из которых содержит часть тренировочной выборки. Эти устройства обычно объединены в звездную топологию, где крайние узлы вычисляют значение оператора исходя из хранящейся локально информации, а сервер агрегирует полученные данные и производит обновление оптимизационных переменных. Отдельно выделяют постановку федеративного обучения \citep{konevcny2016federated, smith2017federated, mcmahan2017communication}. Она подразумевает наличие на каждом устройстве уникальных тренировочных выборок, причем выборки на разных устройствах необязательно одинаково распределены и, как правило, содержат приватную информацию. 

Таким образом, чтобы формально перейти к общей распределенной постановке, представим оператор $F$ в виде конечной суммы:
\begin{align}
\label{eq:finite-sum}
    F(z) = \frac{1}{n}\sum\limits_{i=1}^n F_i(z),
\end{align}
где $n$ -- количество устройств в вычислительной сети, а $F_i(\cdot)$ -- локальный оператор, вычисляемый на основе данных на $i$-ом устройстве. Комбинация \eqref{eq:vi_setup} и \eqref{eq:finite-sum} отражает всевозможный спектр задач распределенного машинного обучения от простых регрессий до нейронных сетей \citep{lecun2015deep}. Несмотря на то, что такой подход применим к классической задачи минимизации (Пример \ref{ex:1}), больший интерес он вызывает для поиска седловых точек (Пример \ref{ex:2}) \citep{beznosikov2020distributed}. На практике это нашло особенно яркое применение в так называемом состязательный подходе, будь то обучение генеративных сетей (GAN) \citep{goodfellow2020generative}, или робастное обучение моделей \citep{liu2020adversarial, zhu2019freelb}. Кроме того, вариационные неравенства также находят широкое применение в различных классических задачах, включая эффективное матричное разложение \citep{bach2008convex}, устранение зернения в изображениях \citep{esser2010general, chambolle2011first}, робастную оптимизацию \citep{ben2009robust}, постановки из экономики, теории игр \citep{von1953theory} и оптимального управления \citep{facchinei2003finite}.

Для решения такой задачи ВН $\eqref{eq:vi_setup} + \eqref{eq:finite-sum}$ необходима адаптация классических методов оптимизации, например, таких как градиентный спуск. Ее можно осуществить по аналогии с Примером \ref{ex:1}, рассмотрев $\nabla f(\cdot)\rightarrow F(\cdot)$. Однако, из-за распределенной постановки \eqref{eq:finite-sum}, чтобы собрать полное значение оператора $F(\cdot)$, вычислительным устройствам необходимо "общаться" друг с другом, как правило, посредством пересылки данных на сервер. Подобные затраты на коммуникацию представляют собой серьезное ограничение распределенных и федеративных подходов. Передача информации часто занимает много времени и нарушает процесс обучения. В некоторых случаях временные затраты могут значительно превышать сложность локальных вычислений, что делает архитектурные решения неэффективными для практического применения. 

Для снижения издержек на обмен информацией между узлами были предложены различные методы уменьшения количества передаваемых данных \citep{kovalev2022optimal, medyakov2023optimal}. В данной работе исследуется применимость техники компрессии к распределенным вариационным неравенствам, сочетая преимущества сжатия и эффективной обработки данных для минимизации коммуникационных затрат и повышения устойчивости алгоритмов.

\section{Обзор литературы}

\subsection{Методы со сжатием}

Для преодоления возникающих проблем с коммуникационными затратами сообщество применяет различные методики. Идея уменьшения количества передаваемой информации в векторах градиентов была исследована в работе \citep{nesterov2012efficiency}. Авторы предлагают модификацию градиентного спуска, при которой на каждом шаге обновляются только некоторые случайные координаты. Такие операторы впоследствии стали называть \textsc{Rand-K} \citep{beznosikov2023biased}. Через несколько лет метод координатного градиентного спуска был адаптирован для распределенной оптимизации \citep{richtarik2016distributed}. Альтернативной методикой борьбы с коммуникационными издержками является техника сжатия передаваемой информации, основанная на использовании квантизации градиентов или параметров моделей \citep{suresh2017distributed}. Данные методы направлены на уменьшение объема данных, за счет ограничения количества бит, необходимых для представления чисел с плавающей точкой \citep{wu2018error, wangni2018gradient, wen2017terngrad}. Обобщение техники сжатия (будь то с помощью выбора координат или квантизации) было сделано в работе \citep{alistarh2017qsgd} посредством введения общего определения оператора сжатия. Авторы предлагают добавить оператор в обычный градиентный спуск.
Однако этот метод не сходится к истинному решению, а лишь к некоторой окрестности, зависящей от дисперсии квантизованных оценок градиентов \citep{gorbunov2021distributed}.


Принимая во внимание данную проблему, сообщество исследователей в области оптимизации и машинного обучения активно разрабатывает распределенные алгоритмы, применяющие технику уменьшения дисперсии. В стандартной нераспределенной задаче стохастической оптимизации она была использована в методах \textsc{SVRG} \citep{johnson2013accelerating}, \textsc{SAG} \citep{roux2012stochastic} and \textsc{SAGA} \citep{defazio2014saga}, \textsc{SARAH} \citep{nguyen2017sarah, beznosikov2024random}, а затем адаптирована для распределенной оптимизации в \textsc{DIANA} \citep{mishchenko2024distributed} путем сжатия не самого градиента, а разности градиентов. В дальнейшем, идея была развита и использована в более продвинутых алгоритмах: \textsc{VR-DIANA} \citep{horvoth2022natural}, \textsc{FedCOMGATE} \citep{haddadpour2021federated}, \textsc{FedSTEPH} \citep{das2022faster}.
Одной из наиболее продвинутых методик применения данных техник в невыпуклой распределенной оптимизации является \textsc{MARINA} \citep{gorbunov2021marina}. В отличие от всех известных подходов, использующих несмещенные операторы сжатия, \textsc{MARINA} базируется на смещенной оценке градиента. Тем не менее, доказано, что данный метод обеспечивает гарантии сходимости, превосходящие все ранее известные техники.

\subsection{Методы решения вариационных неравенств}
Применение различных методов для решения задач вариационных неравенств и задач седловых точек является предметом обширных исследований \citep{juditsky2011solving, gidel2018variational, hsieh2019convergence, mishchenko2020revisiting, hsieh2020explore, gorbunov2022stochastic, beznosikov2023smooth, beznosikov2024first, solodkin2024methods}. Упомянутая выше техника редукции дисперсии была также перенесена и на вариационные неравенства \citep{palaniappan2016stochastic,chavdarova2019reducing,Yura2021,alacaoglu2021stochastic, kovalev2022optimalvi, beznosikov2022unified, pichugin2023optimal, pichugin2024method, medyakovshuffling}. Большинство из этих методов основаны на подходе \textsc{SVRG}, однако некоторые исследования были проведены в анализе более привлекательного с практической точки зрения для задач минимизации метода \textsc{SARAH} \citep{chen2022faster, beznosikov2023sarah}. 

Методы борьбы за эффективность процесса общения также исследовались в общности вариационных неравенств \citep{liu2019decentralized, liu2020decentralized, tsaknakis2020decentralized, beznosikov2020distributed, deng2021local, beznosikov2021distributed, beznosikov2022decentralized}, в том числе и алгоритмы с компрессией. Для липшицевых операторов это было сделано в работах \citep{beznosikov2022distributed, beznosikov2022compression, beznosikov2024similarity} на основе редукции дисперсии из работы \citep{alacaoglu2021stochastic}. 

Говоря о стандартных предположениях, которые используются в анализе методов решения вариационных неравенств, кокоэрсивный режим является одним из наиболее часто встречающихся, несмотря на то, что данное требование является более строгим нежели липшицевость операторов \citep{loizou2021stochastic}. В частности для кокоэрсивных вариационных неравенств существует общий анализ стохастических методов \citep{beznosikov2023stochastic}, включающий и методы со сжатием.

В разрезе работ, использующих предположение кокоэрсивности, стоит выделить основанный на алгоритме \textsc{SARAH} метод \citep{beznosikov2023sarah}. В то же время метод \textsc{MARINA}, как раз и базирующийся на \textsc{SARAH}, еще не были исследованы в контексте вариационных неравенств. Шаг \textsc{MARINA} по факту является шагом \textsc{SARAH} с добавлением сжатия. Как уже отмечалось ранее, метод \textsc{MARINA} являются наиболее привлекательными с точки зрения уменьшения коммуникационных издержек в парадигме распределенного обучения. Цель данной статьи -- исследовать теоретическую и практическую применимость алгоритма \textsc{MARINA} к кокоэрсивным вариационным неравенствам.

\section{Основные результаты}
\begin{itemize}
\item \textbf{Новый метод.} В рамках исследования распределенной постановки представляется новый метод, который использует метод с компрессией \textsc{MARINA} для решения задач вариационных неравенств. 
\item {\bfseries Теоретическая значимость.} В работе представлен полный теоретический анализ предложенного метода для кокоэрсивных монотонных вариационных неравенств.
\item \textbf{Адаптивный анализ.} Данное исследование предлагает несколько возможных расширений. Например, полученные результаты могут быть обобщены для случая произвольного оператора квантования и различных размеров батчей, используемых клиентами.
\item \textbf{Экспериментальная валидация.} Проведены численные эксперименты, которые подчеркивают применимость нашего метода с различными техниками сжатия (квантизацией, выбором координат).
\end{itemize}

\section{Постановка}



Напомним, что рассматривается задачу \eqref{eq:vi_setup}, где оператор $F$ принимает вид \eqref{eq:finite-sum}. Введем следующие предположения на оператор.

\begin{assumption}[Кокоэрсивность] \label{as:coerc}
Каждый оператор $F_i$ является $\ell$-кокоэрсивным, если для любых $u, v \in \mathbb R^d$ выполняется
\begin{equation*}
\label{eq:Lipsh}
\| F_i(u)-F_i(v) \|^2  \leq \ell \langle F_i(u)-F_i(v) , u - v\rangle.
\end{equation*}
\end{assumption}
Это предположение является более сильным аналогом предположения на липшицевость $F_i$. Для задач выпуклой минимизации $\ell$-липшицевость и $\ell$-кокоэрсивность эквивалентны. Сравнение этих предположений для задачи вариационных неравенств приведено в работе \citep{loizou2021stochastic}.

\begin{assumption}[Сильная монотонность]\label{as:strmon}
Оператор $F$ является $\mu$-сильно монотонным, то есть для любых $u, v \in \mathbb R^d$ выполняется
\begin{equation*}
\label{eq:strmon}
\langle F(u) - F(v), u - v \rangle \geq \mu \| u-v\|^2.
\end{equation*}
\end{assumption}
Для задач минимизации это свойство означает сильную выпуклость, а для задач поиска седловой точки -- сильную выпуклость - сильную вогнутость.
\begin{assumption}\label{as:bias}
    Отображение $\mathcal{C}: \mathbb R^d \rightarrow \mathbb R^d$ является несмещенным компрессором, если для любого $u\in\mathbb R^d$ существует $\alpha > 0$ такое, что выполняется
    \begin{align*}
        &\mathbb E\left[\mathcal{C}(u)\right] = u,\\
        &\mathbb E\left\|\mathcal{C}(u) - u\right\|^2 \leqslant \alpha \|u\|^2.
    \end{align*}
\end{assumption}
Это классическое предположение на компрессоры, используемые как в статье \textsc{MARINA} \citep{gorbunov2021marina}, так и в других работах \citep{alistarh2017qsgd, gorbunov2020unified}.
\begin{assumption}\label{as:comp}
    Компрессор $\mathcal{C}(\cdot)$ оставляет долю информации равную $\delta \in (0, 1]$.
\end{assumption}
Например, для компрессоров, которые производят выбор только части координат, это предположение означает, что из $d$ координат, они оставят только $d\delta$.

\section{Метод и анализ сходимости}

Для решения общей постановки задачи вариационных неравенств с липшицевыми операторами обычно используют методы, основанные на подходе \textsc{Extragradient} \citep{juditsky2011solving}. Однако в этой работе рассматриваются кокоэрсивные вариационные неравенства, поэтому достаточно использовать классический подход \textsc{SGD} \citep{robbins1951stochastic, moulines2011non}. Комбинируя его с продвинутой техникой сжатия \textsc{MARINA}, представляем формальное описание рассматриваемого алгоритма (Алгоритм \ref{alg:marina}). Отметим, что наш подход несколько отличается от предложенного в оригинальной статье \citep{gorbunov2021marina}. Там предлагается производить рестарт рекурсивных обновлений аппроксимаций локальных градиентов с некоторой заранее выбранной вероятностью, здесь же это делается раз в фиксированное число итераций, которое называется эпохой (строка \ref{alg:marina:line6} в Алгоритме \ref{alg:marina}).
Далее приведен теоретический анализ сходимости Алгоритма \ref{alg:marina}.

\begin{algorithm}[ht]
	\caption{\foreignlanguage{russian}{\textsc{MARINA} для кокоэрсивных вариационных неравенств}}
	\label{alg:marina}
	\begin{algorithmic}[1]
\STATE
\noindent {\bf Параметры:}  Шаг обучения $\gamma>0$, количество итераций обучения $K$.\\
\noindent {\bf Инициализация:} $z^0 \in \mathbb R^d$.
\FOR {$s=1, 2, \ldots, S$}
    \STATE $z^0 = \tilde z^{s-1}$
    \STATE $g^0 = F(z^0)$
    \STATE \label{alg:marina:line6}$z^1 = z^0 - \gamma g^0$
    \FOR {$k=1, 2, \ldots, K-1$}
        \STATE Отправить $g^{k-1}$ на каждое устройство 
        \FOR {$i = 1, 2, \ldots, n$}
            \STATE $g_i^{k} = g^{k-1} + \mathcal{C}\left(F_i(z^{k}) - F_i(z^{k-1})\right)$
            \STATE Отправить $g_i^k$ на сервер
        \ENDFOR
        \STATE $g^k = \frac{1}{n}\sum\limits_{i=1}^n g_i^k$
        \STATE $z^{k+1} = z^k - \gamma g^k$
    \ENDFOR
    \STATE $\tilde z^s = z^K$
\ENDFOR

\end{algorithmic}
\end{algorithm}

\begin{lemma} \label{lem:1}
Пусть выполнены Предположения \ref{as:coerc}, \ref{as:strmon}, \ref{as:bias}. Тогда для Алгоритма \ref{alg:marina} c $\gamma \leqslant\frac{1}{2\ell(1+\frac{\alpha}{n})}$ верна следующая оценка:
\begin{eqnarray*}
    \mathbb E\| g^{K} \|^2 \leqslant \left(1 - \frac{2\gamma\mu}{3}\right)^{K}\mathbb E\|F(z^0)\|^2.
\end{eqnarray*}
\end{lemma}
\begin{proof}
Зафиксируем любое $s$ в Алгоритме \ref{alg:marina} и рассмотрим обновление аппроксимаций градиента $g^k$:
\begin{eqnarray*}
    \| g^{k} \|^2 &=& 
    \left\| \frac{1}{n}\sum\limits_{i=1}^n g_i^{k} \right\|^2\\
    &=& \left\|g^{k-1} + \frac{1}{n}\sum\limits_{i=1}^n \mathcal{C}\left(F_i(z^{k}) - F_i(z^{k-1})\right)\right\|^2 \\
    &=& \left\|g^{k-1} + \frac{1}{n}\sum\limits_{i=1}^n \left(F_i(z^{k}) - F_i(z^{k-1})\right)\right\|^2 \\
    & & + \left\|\frac{1}{n}\sum\limits_{i=1}^n \left(\mathcal{C}\left(F_i(z^{k}) - F_i(z^{k-1})\right) - \left(F_i(z^{k}) - F_i(z^{k-1})\right)\right)\right\|^2 \\
    & & + 2\Bigl\langle g^{k-1} + \frac{1}{n}\sum\limits_{i=1}^n \left(F_i(z^{k}) - F_i(z^{k-1})\right),\\
    & &\quad\quad\quad\quad\frac{1}{n}\sum\limits_{i=1}^n \left(\mathcal{C}\left(F_i(z^{k}) - F_i(z^{k-1})\right) - \left(F_i(z^{k}) - F_i(z^{k-1})\right)\right)\Bigr\rangle.
\end{eqnarray*}
Пользуясь Предположением \ref{as:bias}, получим, что математическое ожидание правой части скалярного произведении равно нулю. Значит, взяв математическое ожидание, обнуляется скалярное произведение. Более того, используя Предположение \ref{as:bias} для второй нормы, получаем
\begin{eqnarray*}
    & &\left\|\frac{1}{n}\sum\limits_{i=1}^n \left(\mathcal{C}\left(F_i(z^{k}) - F_i(z^{k-1})\right) - \left(F_i(z^{k}) - F_i(z^{k-1})\right)\right)\right\|^2 \\
    & &= \frac{1}{n^2}\sum\limits_{i=1}^n\left\| \mathcal{C}\left(F_i(z^{k}) - F_i(z^{k-1})\right) - \left(F_i(z^{k}) - F_i(z^{k-1})\right)\right\|^2\\
    & & \leqslant \frac{\alpha}{n^2}\sum\limits_{i=1}^n\left\|F_i(z^{k}) - F_i(z^{k-1})\right\|^2,
\end{eqnarray*}
так как все попарные скалярные произведения равны нулю. Таким образом,
\begin{eqnarray*}
    \mathbb E\| g^{k} \|^2 &\leqslant& \mathbb E\|g^{k-1}\|^2 + \left\|F(z^k) - F(z^{k-1})\right\|^2 \\
    & & + 2\left\langle g^{k-1}, F(z^k) - F(z^{k-1})\right\rangle\\
    & & + \frac{\alpha}{n^2}\sum\limits_{i=1}^n\left\|F_i(z^k) - F_i(z^{k-1})\right\|^2\\
    &=& \mathbb E\|g^{k-1}\|^2 + \left\|F(z^k) - F(z^{k-1})\right\|^2 \\
    & & - \frac{2}{\gamma}\left\langle z^{k} - z^{k-1}, F(z^k) - F(z^{k-1})\right\rangle\\
    & & + \frac{\alpha}{n^2}\sum\limits_{i=1}^n\left\|F_i(z^k) - F_i(z^{k-1})\right\|^2 \\
    &=& \mathbb E\|g^{k-1}\|^2 + \left\|F(z^k) - F(z^{k-1})\right\|^2\\
    & & + \frac{\alpha}{n^2}\sum\limits_{i=1}^n\left\|F_i(z^k) - F_i(z^{k-1})\right\|^2 \\
    & & - \frac{2}{3\gamma}\frac{1}{n}\sum\limits_{i=1}^n\left\langle z^{k} - z^{k-1}, F_i(z^k) - F_i(z^{k-1})\right\rangle \\
    & & - \frac{2}{3\gamma}\left\langle z^{k} - z^{k-1}, F(z^k) - F(z^{k-1})\right\rangle \\
    & &- \frac{2}{3\gamma}\left\langle z^{k} - z^{k-1}, F(z^k) - F(z^{k-1})\right\rangle.
\end{eqnarray*}
Для первого и второго скалярного произведения используется Предположение \ref{as:coerc}, а именно $\left\langle z^k - z^{k-1}, F_i(z^k) - F_i(z^{k-1})\right\rangle \geqslant \frac{1}{\ell}\left\|F_i(z^k) - F_i(z^{k-1})\right\|^2$, а для третьего -- Предположение \ref{as:strmon}, а именно \\$\left\langle z^k - z^{k-1}, F(z^k) - F(z^{k-1})\right\rangle \geqslant \mu\left\|F(z^k) - F(z^{k-1})\right\|^2$. Тогда
\begin{eqnarray*}
    \mathbb E\| g^{k} \|^2 &\leqslant& \mathbb E\|g^{k-1}\|^2 + \left(1 - \frac{2}{3\gamma \ell} \right) \left\|F(z^k) - F(z^{k-1})\right\|^2\\ 
    & & + \left(\frac{\alpha}{n} - \frac{2}{3\gamma \ell} \right) \frac{1}{n}\sum\limits_{i=1}^n\left\|F_i(z^k) - F_i(z^{k-1})\right\|^2\\
    & & - \frac{2\mu}{3\gamma}\|z^k - z^{k-1}\|^2.
\end{eqnarray*}
Выберем $\gamma \leqslant \frac{1}{2\ell\left(1 + \frac{\alpha}{n}\right)}$. С учетом этого, получим
\begin{eqnarray*}
    \mathbb E\| g^{k} \|^2 &\leqslant& \left(1 - \frac{2\gamma\mu}{3}\right)\mathbb E\|g^{k-1}\|^2.
\end{eqnarray*}
Запустим рекурсию до первой итерации в эпохе и учтем, что $g^0 = F(z^0)$:
\begin{eqnarray*}
    \mathbb E\| g^{K} \|^2 \leqslant \left(1 - \frac{2\gamma\mu}{3}\right)^{K}\mathbb E\|F(z^0)\|^2.
\end{eqnarray*}
\end{proof}

Следующая лемма дает нам представление о разнице между полным оператором $F(\cdot)$ и его сжатой версией $g$ в течение внутреннего цикла Алгоритма \ref{alg:marina}. 

\begin{lemma} \label{lem:2}
Пусть выполнены Предположения \ref{as:coerc}, \ref{as:strmon}, \ref{as:bias}. Тогда для Алгоритма \ref{alg:marina} верна следующая оценка:
\begin{eqnarray*}
    \mathbb E \left\|F(z^K) - g^{K}\right\|^2 \leqslant \frac{\gamma\ell\left(1+\frac{\alpha}{n}\right)}{1 - \gamma\ell\left(1+\frac{\alpha}{n}\right)} \mathbb E\|F(z^0)\|^2.
\end{eqnarray*}
\begin{proof}
Рассмотрим следующую цепочку:
\begin{eqnarray}
    \notag\mathbb E \left\|F(z^k) - g^k\right\|^2 &=& \mathbb E \left\|\left[F(z^{k-1}) - g^{k-1}\right] + \left[F(z^k) - F(z^{k-1})\right] - \left[g^k - g^{k-1}\right]\right\|^2\\
    \notag&=& \mathbb E \left\|F(z^{k-1}) - g^{k-1}\right\|^2 + \mathbb E \left\|F(z^k) - F(z^{k-1})\right\|^2 \\
    \notag& & + \mathbb E \left\|g^k - g^{k-1}\right\|^2 + 2\mathbb E\left\langle F(z^{k-1}) - g^{k-1}, F(z^k) - F(z^{k-1})\right\rangle \\
    \notag& & - 2\mathbb E\left\langle F(z^{k-1}) - g^{k-1}, g^k - g^{k-1}\right\rangle \\
    \label{l2:eq1}& & - 2\mathbb E\left\langle F(z^k) - F(z^{k-1}), g^k - g^{k-1}\right\rangle.
\end{eqnarray}
Теперь отдельно оценим второе скалярное произведение:
\begin{eqnarray*}
    & &\mathbb E\left\langle F(z^{k-1}) - g^{k-1}, g^k - g^{k-1}\right\rangle \\
    & & \quad\quad= \mathbb E\left\langle F(z^{k-1}) - g^{k-1}, \frac{1}{n}\sum\limits_{i=1}^n\mathcal{C}\left(F_i(z^k) - F_i(z^{k-1})\right)\right\rangle\\
    & & \quad\quad= \mathbb E\left\langle F(z^{k-1}) - g^{k-1}, \frac{1}{n}\sum\limits_{i=1}^n \left(F_i(z^k) - F_i(z^{k-1})\right)\right\rangle \\
    & & \quad\quad\quad+ \mathbb E\Bigl\langle F(z^{k-1}) - g^{k-1},\\
    & & \quad\quad\quad\quad\quad\frac{1}{n}\sum\limits_{i=1}^n\left(\mathcal{C}\left(F_i(z^k) - F_i(z^{k-1})\right) - \left(F_i(z^k) - F_i(z^{k-1})\right)\right)\Bigr\rangle.
\end{eqnarray*}
Так как второе скалярное произведение равно нулю, ввиду Предположения \ref{as:bias} получим
\begin{eqnarray}
\label{l2:eq2}
    \mathbb E\left\langle F(z^{k-1}) - g^{k-1}, g^k - g^{k-1}\right\rangle &=& \mathbb E\left\langle F(z^{k-1}) - g^{k-1}, F(z^k) - F(z^{k-1})\right\rangle.
    \quad~~
\end{eqnarray}
Делая аналогичные выкладки, можно оценить третье скалярное произведение в \eqref{l2:eq1}, как
\begin{eqnarray}
    \notag\mathbb E\left\langle F(z^k) - F(z^{k-1}), g^k - g^{k-1}\right\rangle &=& \mathbb E\left\langle F(z^k) - F(z^{k-1}), F(z^k) - F(z^{k-1})\right\rangle\\
    \label{l2:eq3}&=& \left\|F(z^k) - F(z^{k-1})\right\|^2.
\end{eqnarray}
Подставляя \eqref{l2:eq2} и \eqref{l2:eq3} в \eqref{l2:eq1}, получим
\begin{eqnarray*}
    \mathbb E \left\|F(z^k) - g^k\right\|^2 &=& \mathbb E \left\|F(z^{k-1}) - g^{k-1}\right\|^2 + \mathbb E \left\|F(z^k) - F(z^{k-1})\right\|^2 \\
    & & + \mathbb E \left\|g^k - g^{k-1}\right\|^2 \\
    & & + 2\mathbb E\left\langle F(z^{k-1}) - g^{k-1}, F(z^k) - F(z^{k-1})\right\rangle \\
    & & - 2\mathbb E\left\langle F(z^{k-1}) - g^{k-1}, F(z^k) - F(z^{k-1})\right\rangle \\
    & & - 2\mathbb E\left\|F(z^k) - F(z^{k-1})\right\|^2\\
    &\leqslant& \mathbb E \left\|F(z^{k-1}) - g^{k-1}\right\|^2 + \mathbb E \left\|g^k - g^{k-1}\right\|^2.
\end{eqnarray*}
Запустим рекурсию до первой итерации в эпохе и учтем, что $g^0 = F(z^0)$:
\begin{eqnarray}
\label{l2:eq4}
    \mathbb E\left\|F(z^K) - g^{K} \right\|^2 \leqslant \sum\limits_{k=1}^K\mathbb E\|g^k - g^{k-1}\|^2.
\end{eqnarray}
Пользуясь Предположением \ref{as:bias},
\begin{eqnarray}
    \notag\mathbb E\|g^k - g^{k-1}\|^2 &=& \mathbb E\left\|\frac{1}{n}\sum\limits_{i=1}^n\mathcal{C}\left(F_i(z^k) - F_i(z^{k-1}\right)\right\|^2\\
    \label{l2:eq5}&\leqslant& \frac{1+\frac{\alpha}{n}}{n}\sum\limits_{i=1}^n\left\|F_i(z^k) - F_i(z^{k-1})\right\|^2.
\end{eqnarray}
Далее, аналогично доказательству Леммы \ref{lem:1}:
\begin{eqnarray*}
    \mathbb E\| g^{k} \|^2 &=& \mathbb E\left\|g^{k-1} + \frac{1}{n}\sum\limits_{i=1}^n \mathcal{C}\left(F_i(z^{k}) - F_i(z^{k-1})\right)\right\|^2 \\
    &\leqslant& \mathbb E\left\|g^{k-1} + \frac{1}{n}\sum\limits_{i=1}^n \left(F_i(z^{k}) - F_i(z^{k-1})\right)\right\|^2 \\
    & & + \frac{\alpha}{n^2}\sum\limits_{i=1}^n\left\|F_i(z^{k}) - F_i(z^{k-1})\right\|^2\\
    &\leqslant& \mathbb E\|g^{k-1}\|^2 + \frac{1+\frac{\alpha}{n}}{n}\sum\limits_{i=1}^n\left\|F_i(z^k) - F_i(z^{k-1})\right\|^2\\
    & & + 2\left\langle g^{k-1}, F(z^k) - F(z^{k-1}\right\rangle\\
    &=& \mathbb E\|g^{k-1}\|^2 + \frac{1+\frac{\alpha}{n}}{n}\sum\limits_{i=1}^n\left\|F_i(z^k) - F_i(z^{k-1})\right\|^2 \\
    & & - \frac{2}{\gamma}\left\langle z^{k} - z^{k-1}, F(z^k) - F(z^{k-1}\right\rangle\\
    &=& \mathbb E\|g^{k-1}\|^2 + \frac{1+\frac{\alpha}{n}}{n}\sum\limits_{i=1}^n\left\|F_i(z^k) - F_i(z^{k-1})\right\|^2 \\
    & & - \frac{1}{\gamma}\frac{1}{n}\sum\limits_{i=1}^n\left\langle z^{k} - z^{k-1}, F_i(z^k) - F_i(z^{k-1})\right\rangle \\
    & &- \frac{1}{\gamma}\left\langle z^{k} - z^{k-1}, F(z^k) - F(z^{k-1})\right\rangle\\
    &\leqslant& (1-\gamma\mu)\mathbb E\|g^{k-1}\|^2 \\
    & & + \left(1 - \frac{1}{\gamma\ell\left(1+\frac{\alpha}{n}\right)}\right)\frac{1+\frac{\alpha}{n}}{n}\sum\limits_{i=1}^n\left\|F_i(z^k) - F_i(z^{k-1})\right\|^2.
\end{eqnarray*}
Выражая сумму,
\begin{eqnarray*}
    & &\frac{1+\frac{\alpha}{n}}{n}\sum\limits_{i=1}^n\left\|F_i(z^k) - F_i(z^{k-1})\right\|^2 \\
    & & \quad\quad\quad\quad\leqslant \frac{\gamma\ell\left(1+\frac{\alpha}{n}\right)}{1 - \gamma\ell\left(1+\frac{\alpha}{n}\right)} \left[(1-\gamma\mu)\mathbb E\|g^{k-1}\|^2 - \mathbb E\|g^{k-1}\|^2\right]\\
    & & \quad\quad\quad\quad\leqslant\frac{\gamma\ell\left(1+\frac{\alpha}{n}\right)}{1 - \gamma\ell\left(1+\frac{\alpha}{n}\right)}\left[\mathbb E\|g^{k-1}\|^2 - \mathbb E\|g^{k}\|^2\right].
\end{eqnarray*}
Комбинируя полученную оценку с \eqref{l2:eq5} и \eqref{l2:eq4}, получаем
\begin{eqnarray*}
    \mathbb E \left\|F(z^K) - g^{K}\right\|^2 \leqslant \frac{\gamma\ell\left(1+\frac{\alpha}{n}\right)}{1 - \gamma\ell\left(1+\frac{\alpha}{n}\right)} \mathbb E\|g^{0}\|^2 = \frac{\gamma\ell\left(1+\frac{\alpha}{n}\right)}{1 - \gamma\ell\left(1+\frac{\alpha}{n}\right)} \mathbb E\|F(z^0)\|^2.
\end{eqnarray*}
\end{proof}
\end{lemma}

Теперь объединим результаты Леммы \ref{lem:1} и \ref{lem:2} и получим основную теорему данной статьи.

\begin{theorem} \label{th:1}
Пусть выполнены Предположения \ref{as:coerc}, \ref{as:strmon}, \ref{as:bias}. Тогда для Алгоритма \ref{alg:marina} c $\gamma =\frac{1}{8\ell\left(1+\frac{\alpha}{n}\right)}$ и $K = \frac{30\ell\left(1 + \frac{\alpha}{n}\right)}{\mu}$ верна следующая оценка:
\begin{eqnarray*}
    \mathbb E\| F(\tilde z^s) \|^2 \leqslant \frac{1}{2}\mathbb E\|F(\tilde z^{{s-1}})\|^2.
\end{eqnarray*}
\end{theorem}
\begin{proof}
Начнем с
\begin{eqnarray*}
    \mathbb E\left\| F (z^K) \right\|^2 \leqslant 2 \mathbb E\left\| F(z^K) - g^K \right\|^2 + 2 \mathbb E\left\| g^K \right\|^2.
\end{eqnarray*}
Далее применим Леммы \ref{lem:1}, \ref{lem:2}:
\begin{eqnarray*}
    \mathbb E\left\| F (z^K) \right\|^2 &\leqslant& \left[2\frac{\gamma\ell\left(1+\frac{\alpha}{n}\right)}{1 - \gamma\ell\left(1+\frac{\alpha}{n}\right)} + 2\left(1 - \frac{2\gamma\mu}{3}\right)^K\right]\mathbb E\left\| F (z^0) \right\|^2\\
    &\leqslant& \left[2\frac{\gamma\ell\left(1+\frac{\alpha}{n}\right)}{1 - \gamma\ell\left(1+\frac{\alpha}{n}\right)} + 2\exp\left(-\frac{2}{3}\gamma\mu K\right)\right]\mathbb E\left\| F (z^0) \right\|^2.
\end{eqnarray*}
Здесь использовалось, что $\gamma \mu \in (0;1)$ и $(1 - \gamma \mu)\leq \exp(-\gamma\mu)$. Подставив $\gamma = \frac{1}{8\ell\left(1+\frac{\alpha}{n}\right)}$ и $K = \frac{30\ell\left(1+\frac{\alpha}{n}\right)}{\mu}$, получаем оценку
\begin{align*}
    \mathbb E\left\| F (z^K) \right\|^2 &\leqslant \frac{1}{2} \mathbb E\left\| F (z^0) \right\|^2.
\end{align*}
Так как $z^0 = \tilde z^{s-1}$ и $z^K = \tilde z^s$, была получена сходимость за одну эпоху
\begin{align*}
    \mathbb E\left\| F (\tilde z^s) \right\|^2 &\leqslant \frac{1}{2} \mathbb E\left\| F (\tilde z^{s-1}) \right\|^2.
\end{align*}
\end{proof}

\begin{corollary}\label{cor:1}
Пусть выполнены Предположения \ref{as:coerc}, \ref{as:strmon}, \ref{as:bias}, \ref{as:comp}. Тогда Алгоритм \ref{alg:marina} с $\gamma = \frac{1}{8\ell\left(1+\frac{\alpha}{n}\right)}$ и $K = \frac{30\ell\left(1+\frac{\alpha}{n}\right)}{\mu}$ для достижения $\varepsilon$-точности, где $\varepsilon^2\sim\mathbb E\left\|F(\tilde z^S)\right\|^2$, требует
\begin{align*}
    \mathcal{O}\left( \left[1 + \delta\frac{\ell}{\mu}\left(1 + \frac{\alpha}{n}\right) \right]\log_2 \frac{\| F(z^0)\|^2}{\varepsilon^2}\right) 
\end{align*}
пересылок градиента для каждого устройства.
\end{corollary}
\begin{proof}
Из Теоремы \ref{th:1}:
\begin{eqnarray*}
    \mathbb E\left\| F (\tilde z^S) \right\|^2 &\leqslant (\frac{1}{2})^S \mathbb E\left\| F (z^0) \right\|^2.
\end{eqnarray*}
Тогда для достижении точности $\varepsilon^2\sim\mathbb E\left\|F(\tilde z^S)\right\|^2$ нам нужно следующее число внешних итераций (эпох) S:
\begin{align*}
    S = \mathcal{O}\left( \log_2 \frac{\| F(z^0)\|^2}{\varepsilon^2}\right).
\end{align*}
На каждой внешней итерации полный оператор  пересылается только один раз, а остальные $K-1$ внутренних итераций используется сжатые версии размера $\delta$ (Предположение \ref{as:comp}). Тогда каждое устройство пересылает следующее количество градиентов:
\begin{align*}
    S \times \left(1 + \delta \times (K-1)\right) = \mathcal{O}\left( \left[1 + \delta\frac{\ell}{\mu}\left(1 + \frac{\alpha}{n}\right) \right]\log_2 \frac{\| F(z^0)\|^2}{\varepsilon^2}\right).
\end{align*}
\end{proof}
\begin{remark}
Выбирая $\delta \leqslant \frac{1}{\alpha}$ и $\alpha = n$, Следствие \ref{cor:1} дает следующую оценку на сложность коммуникаций:
\begin{align*}
    \mathcal{O}\left( \left[1 + \frac{\ell}{\mu n} \right]\log_2 \frac{\| F(z^0)\|^2}{\varepsilon^2}\right). 
\end{align*}
\end{remark}

Сравнивая полученный результат с другими методами, отметим, что в работе \citep{beznosikov2023stochastic} была получена такая же оценка сходимости метода \textsc{DIANA} для кокоэрсивных вариационных неравенств.

\section{Эксперименты}

В данном разделе демонстрируется практическое применение предложенного алгоритма, проводя серию экспериментов. Наша основная цель — оценить эффективность различных методов для решения вариационных неравенств с кокоэрсивными свойствами. 
В частности, проводится сравнение производительности метода с несмещенной компрессией, метода с квантизацией \citep{jacob2018quantization} и метода без компрессии для алгоритма \textsc{MARINA}.

Рассмотрим задачу седловой точки конечной суммы, определяемую следующим образом:
\[
    g(x, y) = \frac{1}{n} \sum_{i=1}^n \left[x^\top A_i y + a_i^\top x + b_i^\top y + \frac{\lambda}{2} \|x\|^2 - \frac{\lambda}{2} \|y\|^2\right],
\]
где $A_i \in \mathbb{R}^{d \times d}$, $a_i, b_i \in \mathbb{R}^d$. Данная задача является $\lambda$-сильно выпуклой по $x$ и $\lambda$-сильно вогнутой по $y$, а также $L$-гладкой с $L = \|A\|_2$, где $A = \frac{1}{n} \sum_{i=1}^n A_i$.

Положим $n = 10$, $d = 100$, и $\lambda = 1$. Матрицы $A_i$ и векторы $a_i$, $b_i$ генерируются случайным образом. Примечательно, что константа кокоэрсивности для данной задачи определяется как $\ell = \frac{\|A\|_2^2}{\lambda}$. Эти параметры позволяют исследовать поведение алгоритмов в различных условиях.

В качестве критерия возьмем квадрат нормы градиента на $k-$ой итерации, отнесенного к квадрату нормы функционала на первой итерации: $\left\|\frac{F(z^k)}{F(z^0)}\right\|^2$. Обозначим $\delta = \frac{K}{d}$ -- отношение числа координат в компрессии, которые выбираются, к общей размерности задачи. Для экспериментов использовалась квантизация для перевода чисел из формата fp-32 в формат int-8, что позволяет значительно оптимизировать модели за счет уменьшения размера весов модели в 4 раза и того, что многие процессоры эффективнее обрабатывают 8-битные данные \citep{krishnamoorthi2018quantizing, wu2020integer}. Метод квантизации для удобства на графиках обозначаем \textsc{Q-MARINA}. По оси абсцисс отображается объем передаваемой информации в килобитах, что позволяет оценить эффективность алгоритмов с учетом ограничений на передачу данных. 
Для комплексной оценки методов проектируются три экспериментальных сценария с разными уровнями кокоэрсивности: низким ($\ell \approx 10^2$), средним ($\ell \approx 10^3$) и высоким ($\ell \approx 10^4$).

\begin{figure}[h]
\includegraphics[width=\textwidth]{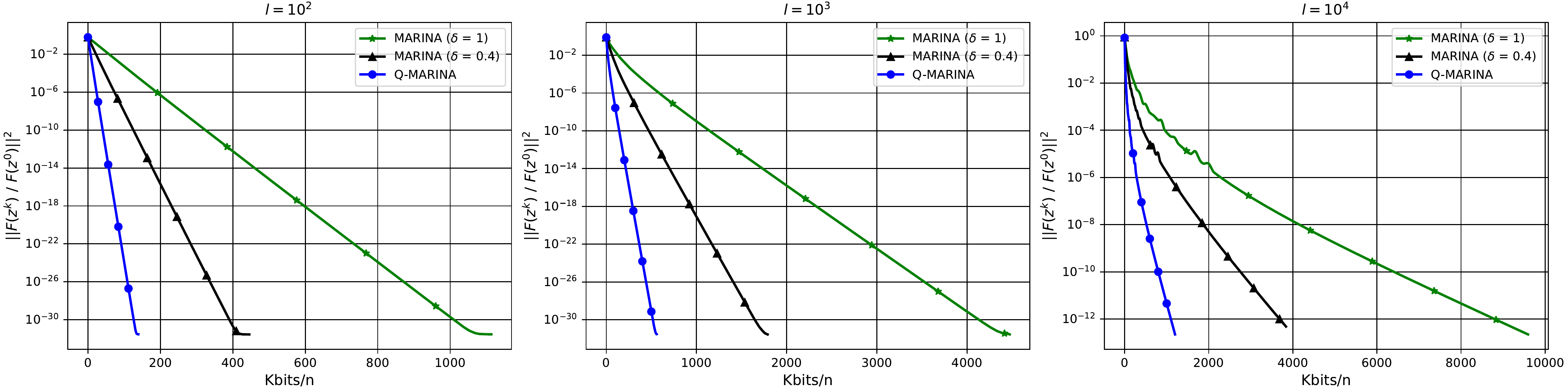}
\caption{\small Сравнение производительности методов на основе MARINA для кокоэрсивных вариационных неравенств на билинейной задаче седловой точки.}
    \label{fig:min}
\end{figure}
Результаты экспериментов представлены на Рис.~\ref{fig:min}. Как видно из графиков, методы \textsc{MARINA} с компрессией стабильно превосходит методы без нее во всех сценариях. В то время как сжатие \textsc{Rand-K} демонстрирует приемлемую производительность, оно уступает квантизации. В свою очередь, метод без сжатия показывает значительно более медленную сходимость с точки зрения объемов передаваемой информации, что подчеркивает его ограничения при решении задач больших размерностей.

Таким образом, эксперименты подтверждают эффективность предложенного подхода к решению вариационных неравенств. Метод \textsc{MARINA} с компрессией обеспечивает значительное сокращение объемов передаваемой информации при сохранении скорости сходимости. Это делает его перспективным инструментом для распределенных систем, где ограниченные ресурсы передачи данных являются критическим фактором.

\section*{Финансирование}

Работа выполнена в Лаборатории проблем федеративного обучения ИСП РАН при поддержке  Минобрнауки России (д.с. № 2 от «19» апреля 2024 г. к Соглашению № 075-03-2024-214 от «18» января 2024 г.)

\bibliographystyle{plainnat}
\bibliography{arxiv}

\end{document}